\documentclass[11pt,a4paper]{article}

\usepackage{epsf,epsfig,amsfonts,amsgen,amsmath,amstext,amsbsy,amsopn,amsthm
}
\usepackage{amsmath,times,mathptmx}
\usepackage{amsfonts,amsthm,amssymb}
\usepackage{amsfonts}
\usepackage{graphics}
\usepackage{latexsym,bm}
\usepackage{amsfonts,amsthm,amssymb,bbding}
\usepackage{indentfirst}
\usepackage{graphicx}
\usepackage{color}
\usepackage[colorlinks=true,anchorcolor=blue,filecolor=blue,linkcolor=blue,urlcolor=blue,citecolor=blue]{hyperref}
\usepackage{float}
\usepackage{tikz}
\evensidemargin=\oddsidemargin\topmargin=-1.5cm

\newtheorem{thm}{Theorem}[section]
\newtheorem{prob}{Problem}[section]

\newtheorem{lem}{Lemma}[section]

\addtocounter{section}{0}

\begin{document}
\title{ Spectral extremal problem on the square of $\ell$-cycle\footnote{Supported by
the APNSF grants, China (Nos. 2108085MA13, 2022AH040151).}}
\author{{\bf Longfei Fang$^{a,b}$}, {\bf Yanhua Zhao$^a$}\thanks{Corresponding author: yhua030@163.com
(Y. Zhao).} \\
\small $^{a}$ School of Mathematics, East China University of Science and Technology, \\
\small  Shanghai 200237, China\\
\small $^{b}$ School of Mathematics and Finance, Chuzhou University, \\
\small  Chuzhou, Anhui 239012, China\\
}

\date{}
\maketitle
{\flushleft\large\bf Abstract}
Let $C_{\ell}$ be the cycle of order ${\ell}$. The square of $C_{\ell}$, denoted by $C_{\ell}^2$, is obtained by joining all pairs of vertices with distance no more than two in $C_{\ell}$.
A graph is called $F$-free if it does not contain $F$  as a subgraph.
Denote by $ex(n,F)$ and $spex(n,F)$ the maximum size and spectral radius over all $n$-vertex $F$-free graphs, respectively.
The well-known  Tur\'{a}n problem asks for the $ex(n,F)$,  and  Nikiforov in 2010  proposed a spectral counterpart, known as Brualdi-Solheid-Tur\'{a}n type problem, focusing on determining $spex(n,F)$.
In this paper, we consider a Tur\'{a}n problem on $ex(n,C_{\ell}^2)$  and a
Brualdi--Solheid--Tur\'{a}n type problem on $spex(n,C_{\ell}^2)$.
We give a sharp bound of $ex(n,C_{\ell}^2)$ and $spex(n,C_{\ell}^2)$ for sufficiently large $n$, respectively.
Moreover, in both results, we characterize the corresponding  extremal graphs for any integer $\ell\geq 6$ that is not divisible by $3$.
\begin{flushleft}
\textbf{Keywords:} Extremal graph; Spectral radius; $C_{\ell}^2$-free
\end{flushleft}
\textbf{AMS Classification:} 05C35; 05C50

\section{Introduction}
 All graphs considered here are  undirected, finite and simple.
Let $G$ be a graph with vertex set $V(G)$ and edge set $E(G)$.
Denote by  $|G|$  the order, $e(G)$ the size, $\nu(G)$ the matching number, $\chi(G)$ the chromatic number, $\delta(G)$ the minimum degree and $\Delta(G)$ the maximum degree, respectively.
For $u,v \in V (G)$,  the distance between $u$ and $v$ is the length of a shortest path between $u$ and $v$.
Let $G^k$ be the $k$-power of a graph $G$, which is obtained by joining all pairs of vertices with distance at most $k$ in $G$. Specially,   the $2$-power of a graph $G$ is also called the square of $G$.
We write $C_l^k$ and $P_l^k$ for the $k$-power of a cycle $C_l$ and a path $P_l$ on $l$ vertices, respectively.

 In this article, we investigate the edge condition and spectral condition for existence of $C_\ell^2$. The study of $C_\ell^k$ has a rich history in graph theory. In 1952, Dirac proved that if $\delta(G)\ge \frac{n}{2}$, then $G$ contains a $C_n$. As a generalization of Dirac's theorem,  P\'{o}sa \cite{Erdos1} in 1963 proposed a conjecture that  if $\delta(G)\ge \frac{2n}{3}$, then $G$ contains a $C_n^2$.
Later,  Seymour \cite{Seymour} in 1974 generalized  P\'{o}sa's conjecture and gave the following conjecture, known as ``P\'{o}sa--Seymour Conjecture", that if $G$ is a connected graph with $\delta(G)\ge \frac{kn}{k+1}$, then $G$ contain a $C_n^k$.
Both of two conjectures  aroused much interest of many scholars, and there were numerous  results on them. As we know, Ch\^{a}u, DeBiasio and Kierstead \cite{CDK} gave the freshest result on P\'{o}sa's conjecture for $n\ge 2\times 10^8,$ and Koml\'{o}s, S\'{a}rk\"{o}zy  and  Szemer\'{e}di \cite{2KSS1998}   verified   P\'{o}sa--Seymour Conjecture  for large $n$.  
For more results, we refer the reader to \cite{CDK,FK1995,FK1996,2FK1996,KSS1996,KSS1998,2KSS1998,LSS2010,RRS,Seymour} and their reference therein.
%
%

Given a graph $H$, we call a graph \emph{$H$}-free
if it does not contain any copy of $H$ as a subgraph. 
Denote by $ex(n,H)$ the maximum number of edges in any $H$-free graph of order $n$.
An $H$-free graph on $n$ vertices with $ex(n,H)$ edges is called an  \emph{extremal graph}  with respect to $ex(n,H)$.
The classic Tur\'{a}n's problem, which focuses on determining $ex(n,H)$, is one of the central problems in extremal graph theory.
Tur\'{a}n's problem and many kinds of its variations have been paid much attention and a considerable number of influential results in extremal graph theory have been obtained (see for example, a survey \cite{FS2013}).
In particular,  Ore \cite{O1961} in 1961 determined $ex(n,C_n)$. Recently, Khan and Yuan \cite{KY2022}  gave $ex(n,C_n^2)$.  The well--known Mantel Theorem showed  $ex(n,C_{3})$, and Zoltan and David \cite{FG-2015} described $ex(n,C_{2l+1})$ for $l\ge 2$. The freshest result on even cycle $ex(n,C_{2l+2})$ was proved in \cite{HS2021}. Moreover,  for each $\ell\in \{3,4,5\}$,  we see that $C_{\ell}^2\cong K_{\ell}$, and so   $ex(n,C_{\ell}^2)$  was determined by Tur\'{a}n \cite{Turan} in a more general context. Therefore, it is a natural wish to consider the following problem.
\begin{prob}\label{prob-01}
For any integer $\ell\geq 6$, what is the exact value of $ex(n,C_l^2)$?
\end{prob}
In this paper, we first give an answer to Problem \ref{prob-01}.
Given two graphs $G$ and $H$, let $G\cup H$ be the disjoint union of $G$ and $H$, and $G+H$ be the  graph obtained from  $G\cup H$ by adding all edges between $G$ and  $H.$ Denote by $T_{n,r}$ the \emph{Tur\'{a}n graph}, which  is the complete $r$-partite graph on $n$ vertices
where each partite set has either $\lceil\frac{n}{r}\rceil$ or $\lfloor\frac{n}{r}\rfloor$ vertices.
For short, we set $G(n)=K_1+T_{n-1,3}$.
 \begin{thm}\label{theorem1.03}
Let $\ell\geq 6$ and $G$ be an $n$-vertex $C_{\ell}^2$-free graph for sufficiently large $n$.
Then  $e(G)\leq e(G(n))$, with equality if and only if $G\cong G(n)$ and $\ell=3k+2$ for some integer $k\geq 2$.
\end{thm}

Given a graph $G$, let $A(G)$ be its  adjacency matrix, and $\rho(G)$ be its spectral radius.
For any graph $H$, we define
$$spex(n,H)=\max\{\rho(G): H\nsubseteq G, |V(G)|=n\}.$$
An $H$-free graph on $n$ vertices with spectral radius $spex(n,H)$ is called an \textit{extremal graph} with respect to $spex(n,H)$.
 In 2010, Nikiforov \cite{Nikiforov1} proposed a spectral counterpart as follows, which is known as Brualdi-Solheid-Tur\'{a}n type problem.
 \begin{prob}[\cite{Nikiforov1}]\label{pro1}
For a given graph $H$, what is the exact value of $spex(n,H)$?
 \end{prob}
In the past decades, the problem  has been investigated by many researchers  for many specific graphs $H$.
We refer the reader to the three surveys \cite{CZ2018,LLF2022,N11} and the reference therein for more detail.
Specially,
Nikiforov \cite{Nikiforov2} determined $spex(n,C_{2\ell+1})$ for sufficiently large $n$.
Nikiforov \cite{Nikiforov5} and Zhai et al. \cite{Zhai-3}
determined the unique extremal graph with respect to
$spex(n,C_4)$ for odd and even $n$, respectively.
 Subsequently, Zhai and Lin \cite{ZL2020} determined the unique extremal graph with respect to $spex(n,C_6)$. 
Very recently, Cioab\u{a}, Desai and Tait \cite{Cioaba1} determined the unique extremal graph with respect to $spex(n,C_{2\ell})$ for $\ell\geq3$ and $n$ large enough.
  Yan et al. \cite{Yan} determined the unique extremal graph with respect to $spex(n,C_n^2)$ for $n\geq15$, and they further guessed the extremal graph with respect to $spex(n,C_n^k)$ for $n$ large enough.
Moreover, for each $\ell\in \{3,4,5\}$,   $C_{\ell}^2\cong K_{\ell}$, and so  $spex(n,C_{\ell}^2)$ was determined by  Nikiforov \cite{Nikiforov5} in a more general context.
Then, it is a natural wish to consider a spectral counterpart on Problem \ref{prob-01}.
Thus, we obtain the following result.

\begin{thm}\label{theorem1.3}
Let $\ell\geq 6$ and $G$ be an $n$-vertex $C_{\ell}^2$-free graph for sufficiently large $n$.
Then  $\rho(G)\leq \rho(G(n))$, with equality if and only if $G\cong G(n)$ and $\ell=3k+2$ for some integer $k\geq 2$.
\end{thm}
It is worth mentioning that the extremal graph of Theorem \ref{theorem1.03} is the same as the one in Theorem \ref{theorem1.3}, which implies that they give a new example illustration for which an edge-extremal graph is also a spectral-extremal graph.

\section{Proof of Theorem \ref{theorem1.03}}\label{section04}
Before beginning our proof, we first give some  notations  not defined above. Let $G$ be a graph,  and let  $d_G(v)$ be  the degree of the vertex $v$ in $G$.
 Given a graph $G$, for a vertex $v\in V(G)$ and  a vertex subset $X\subseteq V(G)$ (possibly $v\notin X$),  we denote  $N_X(v)$   the set of neighbors of $v$ in $X$ and $d_X(v)=|N_X(v)|$, and let $G[X]$ be the subgraph of $G$ induced by $X$ and $G-X$ be the subgraph of $G$ induced by $V(G)\setminus X$. Similarly,
For an edge subset $Y\in E(G)$, let $G-Y$ be the graph obtained from $G$ by deleting all edges in $Y$.

 A graph $H$ is called an \emph{edge-color-critical graph}, if there exists an edge $e\in E(H)$
such that $\chi(H-\{e\})<\chi(H)$.
Simonovits  proved the following two results, which are  important tools in extremal graph theory.

 \begin{lem}\label{lemma2.001}\emph{(\cite{Simonovits1966})}
Let $r\geq 2$ and $F$ be an edge-color-critical graph with $\chi(F)=r+1$.
Then $T_{n,r}$ is the unique extremal graph with respect to $ex(n,F)$ for sufficiently large $n$.
 \end{lem}

 For convenience, we cite the following lemma, which was given by Simonovits \cite{Simonovits} in a more general result. 
\begin{lem}\label{lemma2.002}\emph{(\cite{Simonovits})}
Let $F$ be a given graph with $\chi(F)=r+1$.
If omitting any $s$ vertices of $F$ we obtain a graph with chromatic number greater than $r$ but omitting $s+1$ suitable edges of $F$
we get a $r$-chromatic graph, then $K_s+T_{n-s,r}$ is the unique extremal graph with respect to $ex(n,F)$ for sufficiently large $n$.
\end{lem}

Also, Erd\H{o}s and Simonovits \cite{Erdos-Simonovits} proved the following lemma.
\begin{lem} \label{lemma2.003} \emph{(\cite{Erdos-Simonovits})}
Let $F$ be a graph with chromatic number $\chi(F)=r+1$.
For sufficiently large $n$,
\begin{center}
  $ex(n,F)=(1-\frac{1}{r})\binom{n}{2}+o(n^2).$
\end{center}
\end{lem}

Given a graph $F$.
We say that $\{V_1,V_2,\dots,V_{\chi(F)}\}$ is a \emph{good-partition} of $V(F)$ if
$\{V_1,V_2,\dots,V_{\chi(F)}\}$ is a partition of $V(F)$ and for each $i\in\{ 1, 2,\dots, \chi(F) \}$, $V_i$ is an independent set of $F$.
Now we  give the following lemma on the chromatic number of $C_{\ell}^2$ which will be crucial for the sequel.
\begin{lem}\label{lemma2.004}
Let $C_{\ell}=v_1v_2\dots v_{\ell}v_1$. \\
(i) If $\ell\equiv 0 \pmod{3}$, then $\chi(C_{\ell}^2)=3$.\\
(ii) If $\ell\equiv 1 \pmod{3}$, then $\chi(C_{\ell}^2)=4$ and $\chi(C_{3\ell+1}^2-\{v_1v_{\ell}\})=3$.\\
(iii) If $\ell\equiv 2 \pmod{3}$, then $\chi(C_{\ell}^2)=4$.
Moreover, $\chi(C_{\ell}^2-\{v\})=4$ for any vertex $v\in V(C_{\ell}^2)$, and $\chi(C_{\ell}^2-\{v_{\ell}v_1,v_3v_4\})=3$.
\end{lem}
 \begin{proof}
Let $P_{\ell}=v_1v_2\dots v_{\ell}$.
Recall that $P_{\ell}^2$ is obtained by joining all pairs of vertices with distance at most  two in $P_{\ell}$.
Then, $\{v_1,v_2,v_3\}$ induces a triangle in $P_{\ell}^2$,
which indicates that $\chi(P_{\ell}^2)\geq 3$.
We first show that $\chi(P_{\ell}^2)=3$.
For every $i\in \{1,2,3\}$, set
        $$V_i=\{v_j~|~j\in\{1,\dots,\ell\}~~\text{and}~(j-i)~\text{is divisible by}~3\}.$$
Then, any two vertices $v_{i_1}$ and $v_{i_2}$ in $V_i$ are non-adjacent in $P_{\ell}^2$ since  $|i_1-i_2|\geq 3$.
Thus, $\chi(P_{\ell}^2)=3$.
Moreover, since $V_i$ is an independent set of $P_{\ell}^2$, $\{V_1,V_2,V_3\}$ is a good-partition of $V(P_{\ell}^2)$.
Furthermore, we shall prove that $\{V_1,V_2,V_3\}$ are the unique good-partition of $V(P_{\ell}^2)$.
In some good-partition, for any $i\in \{4,\dots,\ell\}$,
since $\{v_{i-3},v_{i-2},v_{i-1}\}$ induces a triangle in $P_{\ell}^2$,
we can see that $v_{i-3},v_{i-2},v_{i-1}$ are in different parts.
Similarly, $v_{i-2},v_{i-1},v_{i}$ are also in different parts.
This implies that $v_{i}$ and $v_{i-3}$ are in the same part.
Furthermore, $v_i$ and $v_j$ are in the same part if and only if $(j-i)$ is divisible by 3,
and hence $\{V_1,V_2,V_3\}$ are the unique good-partition of $V(P_{\ell}^2)$.

Note that $P_{\ell}^2$ is a subgraph of $C_{\ell}^2$. Then, $\chi(C_{\ell}^2)\geq \chi(P_{\ell}^2)=3$.
Now assume that $\chi(C_{\ell}^2)=3$.
Then, $C_{\ell}^2$ is also a $3$-partite graph with the unique good-partition $\{V_1,V_2,V_3\}$.

(i) If $\ell\equiv 0 \pmod{3}$, then for every $i\in \{1,2,3\}$, $V_i$ is an independent set of $C_{\ell}^2$.
So, $\chi(C_{\ell}^2)=3$.

(ii) If $\ell\equiv 1 \pmod{3}$, then $v_1,v_{\ell}\in V_1$.
However, $v_1v_{\ell}\in E(C_{\ell}^2)$, contradicting the definition of good-partition of $V(C_{\ell}^2)$.
Thus, $\chi(C_{\ell}^2)\geq 4$.
Set $V_1'=V_1\setminus\{v_{\ell}\}$, $V_2'=V_2$, $V_3'=V_3$ and $V_4'=\{v_{\ell}\}$.
Then each $V_i'$ is an independent set of $C_{\ell}^2$, which implies that $\chi(C_{\ell}^2)=4$.
Now we consider the graph $C_{\ell}^2-\{v_1v_{\ell}\}$.
Note that $P_{\ell}^2$ is a subgraph of $C_{\ell}^2-\{v_1v_{\ell}\}$.
Then, $\chi(C_{\ell}^2-\{v_1v_{\ell}\})\geq \chi(P_{\ell}^2)=3$.
Then for every $i\in \{1,2,3\}$, $V_i$ is an independent set of $C_{\ell}^2$,
which implies that $\chi(C_{\ell}^2-\{v_1v_{\ell}\})=3$, as desired.

(iii) If $\ell\equiv 2 \pmod{3}$, then $v_1,v_{\ell-1}\in V_1$.
However, $v_1v_{\ell-1}\in E(C_{\ell}^2)$, contradicting  the definition of good-partition of $V(C_{\ell}^2)$.
Thus, $\chi(C_{\ell}^2)\geq 4$.
Set $V_1'=(V_1\setminus\{v_1,v_4\})\cup\{v_3\}$,
$V_2'=(V_2\setminus\{v_2,v_{\ell}\})\cup\{v_1\}$,
$V_3'=(V_3\setminus\{v_3\})\cup \{v_2\}$ and $V_4'=\{v_4,v_{\ell}\}$.
Then each $V_i'$ is an independent set of $C_{\ell}^2$, which implies that $\chi(C_{\ell}^2)=4$.

Now we show that $\chi(C_{\ell}^2-\{v\})=4$ for any vertex $v\in V(C_{\ell}^2)$.
Clearly, $\chi(C_{\ell}^2-\{v\})\leq \chi(C_{\ell}^2)=4$.
Thus, it suffices to prove that  $\chi(C_{\ell}^2-\{v\})\ge 4$.
Suppose  to the contrary that $\chi(C_{\ell}^2-\{v\})\le 3$.
Note that $\chi(C_{\ell}^2-\{v\})\ge 3$ since $C_3\subseteq C_{\ell}^2-\{v\}.$
Therefore, $\chi(C_{\ell}^2-\{v\})= 3$.
We may assume without loss of generality that $v=v_{\ell}$.
By the similar discussion of $P_{\ell}^2$, we can obtain that $v_1$ and $v_{\ell-1}$ are in the same part,
which is contrary to  $v_1v_{\ell-1}\in E(C_{\ell}^2-\{v_{\ell}\})$.
Thus, $\chi(C_{\ell}^2-\{v\})=4$.

Set $V_1''=(V_1\setminus\{v_1\})\cup\{v_3\}$,
$V_2''=(V_2\setminus\{v_2\})\cup\{v_1\}$ and
$V_3''=(V_3\setminus\{v_3\})\cup \{v_2\}$.
Then each $V_i''$ is an independent set of $C_{\ell}^2-\{v_{\ell}v_1,v_3v_4\}$, which implies that $\chi(C_{\ell}^2-\{v_{\ell}v_1,v_3v_4\})=3$.
\end{proof}

Now we are in a position to give the  proof of Theorem \ref{theorem1.03}.
\vspace{2mm}

\begin{proof}[\textbf{Proof of Theorem~\ref{theorem1.03}}]
If $\ell\equiv 0 \pmod{3}$, then by Lemma \ref{lemma2.004}, $\chi(C_{\ell}^2)=3$.
Furthermore,  by Lemma \ref{lemma2.003} and a simple calculation, we have $e(G)=\frac{1}{2}\binom{n}{2}+o(n^2)<e(G(n))$, as desired.
If $\ell\equiv 1 \pmod{3}$,
then by Lemma \ref{lemma2.004} (ii), $C_{\ell}^2$ is an edge-critical graph with $\chi(C_{\ell}^2)=4$.
By Lemma \ref{lemma2.001}, $T_{n,3}$ is the unique extremal graph with respect to $ex(n,C_{\ell}^2)$ for sufficiently large $n$, and by a direct calculation, we yield that $e(T_{n,3})<e(G(n))$, as desired.
It remains the case $\ell\equiv 2 \pmod{3}$. According to Lemmas \ref{lemma2.002} and \ref{lemma2.004} (iii),
we obtain that $G(n)$ is the unique extremal graph with respect to $ex(n,C_{\ell}^2)$, as desired.
\end{proof}

\section{Proof of Theorem \ref{theorem1.3}}\label{section4}
In this section, we first list some lemmas that will be used in later proof of  Theorem  \ref{theorem1.3}.
Denote by $K_r(n_1,n_2,\dots,n_r)$ the complete $r$-partite graph with classes of orders $n_1,n_2,\dots,n_r$.
The following is the spectral version of the Stability Lemma due to Nikiforov \cite{Nikiforov4}.

\begin{lem}\label{lemma3.001}\emph{(\cite{Nikiforov4})}
Let $r\ge 2$, $\frac{1}{\ln n}<c<r^{-8(r+21)(r+1)}$, $0<\varepsilon<2^{-36}r^{-24}$ and $G$ be an $n$-vertex graph.
If $\rho(G)>(1-\frac1r-\varepsilon)n$, then one of the following holds:\\
(i) $G$ contains a $K_{r+1}(\lfloor c\ln n\rfloor, \dots,\lfloor c\ln n\rfloor,\lceil n^{1-\sqrt{c}}\rceil)$;\\
(ii) $G$ differs from $T_{n,r}$ in fewer than $(\varepsilon^{\frac{1}{4}}+c^{\frac{1}{8r+8}})n^2$ edges.
\end{lem}

From the above lemma,
Desai et al. \cite{Desai} obtained the following result,
which help us to  present an approach to prove Theorem \ref{theorem1.3}.

\begin{lem} \label{lemma3.002}\emph{(\cite{Desai})}
Let $F$ be a graph with chromatic number $\chi(F)=r+1$.
For every $\varepsilon>0$, there exist $\delta>0$ and $n_0$ such that
if $G$ is an $F$-free graph on $n\ge n_0$ vertices with $\rho(G)\ge (1-\frac1r-\delta)n$,
then $G$ can be obtained from $T_{n,r}$ by adding and deleting at most $\varepsilon n^2$ edges.
\end{lem}

By Nikiforov's result in \cite{Nikiforov10}   and a more careful calculation on the
equality case in his proof, one can get the following spectral version of  the edge-color-critical theorem as follows.

\begin{lem} \label{lemma3.003}\emph{(\cite{Nikiforov10})}
Let $r\geq 2$ and $F$ be an edge-color-critical graph with $\chi(F)=r+1$.
Then $T_{n,r}$ is the unique extremal graph with respect to $spex(n,F)$ for sufficiently large $n$.
\end{lem}

%
Now we  give the  proof of Theorem \ref{theorem1.3}.
In what follows, assume that $G$ is an extremal graph with respect to $spex(n,C_{\ell}^2)$.
Clearly, $G$ is connected. Otherwise, we choose $G_1$ and $G_2$ as two components, where $\rho(G_1)=\rho(G),$ and then we can add a cut edge between $G_1$ and $G_2$ to obtain a new graph with lager spectral radius, which is a contradiction.
By Perron-Frobenius theorem, there exists a positive unit eigenvector
$X=(x_1,\ldots,x_n)^T$ corresponding to $\rho(G)$.

If $\ell\equiv 1 \pmod{3}$,
then by Lemma \ref{lemma2.004} (ii), $C_{\ell}^2$ is an edge-critical graph with $\chi(C_{\ell}^2)=4$.
By Lemma \ref{lemma3.003}, $T_{n,3}$ is the unique extremal graph with respect to $spex(n,C_{\ell}^2)$ for sufficiently large $n$. Clearly,  $\rho(T_{n,3})<\rho(G(n))$, as desired.

If $\ell\equiv 0 \pmod{3}$, then by Lemma \ref{lemma2.004}, $\chi(C_{\ell}^2)=3$.
Note that $T_{n,2}$ is $C_{\ell}^2$-free.
Then $\rho(G)\geq \rho(T_{n,2})\geq (\frac12-\delta)n$ for any $\delta>0$.
Setting $\varepsilon=\frac{1}{10^8}$, by Lemma \ref{lemma3.002} we have
\begin{equation}\label{eq-1}
  e(G)\leq e(T_{n,2})+\frac{1}{10^8} n^2.
\end{equation}
Now we prove that $\rho(G)<\rho(G(n))$.
Otherwise, since $\rho(G(n))\geq \frac{2}{3}n$ (see Lemma \ref{lemma3.004} as below), we have $\rho(G)\geq \rho(G(n))\geq \frac{2}{3}n$.
By Lemma \ref{lemma3.001}, one of the following holds:
\begin{itemize}
  \item[(i)] $G$ contains a $K_{4}(\lfloor c\ln n\rfloor,\lfloor c\ln n\rfloor, \lfloor c\ln n\rfloor,\lceil n^{1-\sqrt{c}}\rceil)$;
  \item[(ii)] $G$ differs from $T_{n,3}$ in fewer than $(\varepsilon^{\frac{1}{4}}+c^{\frac{1}{24}})n^2\leq \frac{1}{50}n^2$ edges.
\end{itemize}
If $(i)$ holds, then $G$ contains a  $C_{\ell}^2$ for sufficiently large $n$, a contradiction;
if $(ii)$ holds, then $e(G)\geq e(T_{n,3})-\frac{1}{50}n^2$, contrary to (\ref{eq-1}).
Therefore, $\rho(G)< \rho(G(n))$, as desired.

Therefore, we are left with  considering the situation $\ell\equiv 2 \pmod{3}$.
In the following, we always assume that $\ell=3k+2$ with $k\geq 2$.
Set $x_{u^*}=\max\{x_i~|~i\in V(G)\}$, and  choose a positive constant $\eta<\frac{1}{9(120k+48)}$, which will be frequently used in the sequel.
Our goal is to  prove $G\cong G(n)$ for $n$ sufficiently large.

\begin{lem}\label{lemma3.004}
$\rho(G)\geq \rho(G(n)) \geq \frac{2n}{3}.$
\end{lem}

\begin{proof}
Recall that $G(n)\cong K_1+T_{n-1,3}$. From Theorem~\ref{theorem1.03} we know that $G(n)$ is $C_{3k+2}^2$-free.
Since $e(T_{n-1,3})=\big\lfloor\frac{(n-1)^2}{3}\big\rfloor\geq \frac{n^2-2n}{3},$
we have
\begin{center}
  $e(G(n))=e(T_{n-1,3})+(n-1)\ge \frac{1}{3}n^2+\frac{1}{3}n-1. $
\end{center}
Using the Rayleigh quotient gives
\begin{center}
  $\rho(G)\ge \rho(G(n))\ge \frac{\mathbf{1}^TA(G(n))\mathbf{1}}{\mathbf{1}^T\mathbf{1}}=\frac{2e(G(n))}{n}\ge \frac{2n}{3}+\frac{2}{3}-\frac{2}{n}\geq \frac{2n}{3}$
\end{center}
for sufficiently large $n$, as desired.
\end{proof}

\begin{lem}\label{lemma3.005}
For $n$ sufficiently large, $e(G)\ge \big(\frac{1}{3}-\eta^3\big)n^2.$
Furthermore, $G$ admits a partition $V(G)=V_1\cup V_2\cup V_3$ such that $\sum_{1\leq i<j\leq 3}e(V_i,V_j)$ attains the maximum,
$\sum_{i=1}^{3}e(V_i)\le \eta^3 n^2$ and $\big||V_i|-\frac{n}{3}\big|\le\eta n$ for each $i\in \{1,2,3\}.$
\end{lem}

\begin{proof}
Let $\varepsilon$ be a positive constant with $\varepsilon<\eta^3$.
From Lemma \ref{lemma2.004} we have $\chi(C_{3k+2}^2)=4$, and
 Lemma \ref{lemma3.004} gives $\rho(G)\geq \frac{2n}{3}$.
Combining these with Lemma \ref{lemma3.002}, we have
$e(G)\geq \frac{1}{3}n^2-\eta^3 n^2$,
and there exists a partition $V(G)=U_1\cup U_2\cup U_3$ such that
$\lfloor\frac{n}{3}\rfloor\le |U_1|\leq|U_2|\leq |U_3| \leq \lceil\frac{n}{3}\rceil$ and $\sum_{i=1}^3e(U_i)\le \eta^3 n^2$.

We now select a new partition $V(G)=V_1\cup V_2\cup V_3$
such that $\sum_{1\leq i<j\leq 3}e(V_i,V_j)$ attains the maximum.
Then $\sum_{i=1}^{3}e(V_i)$ attains the minimum, and so
\begin{center}
$\sum\limits_{i=1}^{3}e(V_i)\le \sum\limits_{i=1}^{3}e(U_i)\le \eta^3 n^2.$
\end{center}
On the other hand, assume that $|V_1|=\frac{n}{3}+a$, then
$|V_2||V_3|\leq \left(\frac{|V_2|+|V_3|}{2}\right)^2=\frac{1}{4}\left(\frac{2n}{3}-a\right)^2$.
Thus,
\begin{center}
   $\sum\limits_{1\leq i<j\leq 3}e(V_i,V_j)= |V_1|(|V_2|+|V_3|)+|V_2||V_3|\leq\frac{n^2}{3}-\frac{3}{4}a^2.$
\end{center}
Therefore,
\begin{center}
  $e(G)= \sum\limits_{1\leq i<j\leq 3}e(V_i,V_j)+\sum\limits_{i=1}^{3}e(V_i)\le \frac{n^2}{3}-\frac{3}{4}a^2+\eta^3 n^2.$
\end{center}
Combining this with $e(G)\ge \frac{1}{3}n^2-\eta^3 n^2$, we get $a^2\leq \frac{8}{3}\eta^{3} n^2\leq \eta^2 n^2$,
and so $||V_1|-\frac{n}{3}|=|a|\leq\eta n$.
Similarly, $||V_i|-\frac{n}{3}|\leq \eta n$ for $i\in \{2,3\}$.
This completes the proof.
\end{proof}

In the following, we shall define two vertex subsets $S$ and $W^{\lambda}$ of $G$.

\begin{lem}\label{lemma3.006}
Let $S=\{v\in V(G)~|~d_G(v)\le \big(\frac{2}{3}-6\eta\big)n\}.$
Then $|S|\le \eta n$.
\end{lem}

\begin{proof}
Suppose to the contrary that $|S|>\eta n$,
then there exists a subset $S'\subseteq S$ with $|S'|=\lfloor\eta n\rfloor$.
Set $n'=|G-S'|=n-\lfloor\eta n\rfloor$. Then $n'+1< (1-\eta)n+2$.
Combining these with Lemma \ref{lemma3.005}  that $e(G)\ge  (\frac{1}{3}-\eta^3\big)n^2$,  we deduce that
\begin{align*}
 e(G-S')&\ge  e(G)-\sum_{v\in S'}d_G(v)\nonumber\\
  &\ge  \big(\frac{1}{3}-\eta^3\big)n^2-\eta n\Big(\frac{2}{3}-6\eta\Big)n\nonumber\\
  &= \frac13\big(1-3\eta^3-2\eta+18\eta^2\big)n^2\nonumber\\
  &>  \frac13\big(n'+1\big)^2
\end{align*}
for sufficiently large $n$.
Note that $\frac13(n'+1)^2>e(G(n')).$
Then, $e(G-S')>e(G(n'))$. By Theorem \ref{theorem1.03},
$G-S'$ contains a copy of $C_{3k+2}^2$,
contradicting that $G$ is $C_{3k+2}^2$-free.
\end{proof}

\begin{lem}\label{lemma3.6}
Let $W^{\lambda}=W_1^{\lambda}\cup W_2^{\lambda}\cup W_3^{\lambda}$, where $W_i^{\lambda}=\{v\in V_i~|~d_{V_i}(v)\ge 2^{\lambda}\eta n\}$.
Then $|W^{\lambda}|\le \frac{1}{2^{\lambda-1}}\eta^2 n$.
\end{lem}

\begin{proof}
For $i\in \{1,2,3\}$,
\begin{center}
  $2e(V_i)=\sum\limits_{v\in V_i}d_{V_i}(v)\ge
\sum\limits_{v\in W_i^{\lambda}}d_{V_i}(v)\ge |W_i^{\lambda}|\cdot 2^{\lambda}\eta n.$
\end{center}
Combining this with Lemma \ref{lemma3.005} gives
\begin{center}
  $\eta^3 n^2\ge \sum\limits_{i=1}^{3}e(V_i)\ge 2^{\lambda-1} \sum\limits_{i=1}^{3}|W_i^{\lambda}|\eta n=2^{\lambda-1}|W^{\lambda}|\eta n.$
\end{center}
This yields that $|W^{\lambda}|\leq \frac{1}{2^{\lambda-1}}\eta^2 n$.
\end{proof}

For every $i\in \{1,2,3\}$, denote by  $V_i^{\lambda}=V_i\setminus (W^{\lambda}\cup S)$. Then we give the following result.

\begin{lem}\label{lemma3.7}
Let $i_1,i_2,i_3$ be three distinct integers in  $\{1,2,3\}$, $s\leq 3k+1$, and $\lambda\in \{1,5\}$. Then \\
(i) for any $u\in V_{i_1}\setminus S$, $d_{V_{i_2}}(u)\geq\big(\frac{1}{9}-5\eta\big)n$; \\
(ii) for any $u\in V_{i_1}^{\lambda}$, $d_{V_{i_2}}(u)\geq\left(\frac{1}{3}-(7+2^{\lambda})\eta\right)n$;\\
(iii) if $u\in W_{i_1}^5$ and $\{u_1,u_2\}\subseteq V_{i_2}^{1}\cup V_{i_3}^{1}$,
then there exist at least $3k+3$ vertices in $V_{i_1}^1$ adjacent to $u,u_1$ (and $u_2$);\\
(iv) if $u\in  (W_{i_1}^{\lambda}\cup W_{i_3}^{\lambda}) \setminus S$ and
    $\{u_1,\dots,u_s\}\subseteq V_{i_1}^{\lambda}\cup V_{i_3}^{\lambda}$,
then there exist at least $3k+3$ vertices in $V_{i_2}^{\lambda}$ adjacent to $u_1,\dots,u_s$ (and $u$).
\end{lem}

\begin{proof}
For any vertex $v\notin S$, we have $d_G(v)>\big(\frac{2}{3}-6\eta\big)n$.

(i) Since $V(G)=V_1\cup V_2\cup V_3$ is a partition such that $\sum_{1\leq i<j\leq 3}e(V_i,V_j)$ attains the maximum,
$d_{V_{i_1}}(u)\leq \frac{1}{3}d_{G}(u)$.
By Lemma \ref{lemma3.005}, $|V_{i_3}|\leq \big(\frac{1}{3}+\eta\big)n$.
These, together with $d_G(u)> \big(\frac23-6\eta\big)n$, give that
\begin{center}
  $d_{V_{i_2}}(u)
    = d_G(u)-d_{V_{i_1}}(u)-d_{V_{i_3}}(u)
    \geq \frac{2}{3}d_G(u)- \big(\frac{1}{3}+\eta\big)n
    \geq \big(\frac{1}{9}-5\eta\big)n. $
\end{center}

(ii) Since $u\notin S$, $d_G(u)> \big(\frac{2}{3}-6\eta\big)n$.
Moreover, since $u\notin W_{i_1}^{\lambda}$, we have $d_{V_{i_1}}(u)< 2^{\lambda}\eta n$.
Then it follows from $d_{V_{i_3}}(u)\leq |V_{i_3}|\leq \big(\frac{1}{3}+\eta\big)n$ that
\begin{center}
  $ d_{V_{i_2}}(u)= d_G(u)-d_{V_{i_1}}(u)-d_{V_{i_3}}(u)
            \geq \left(\frac{1}{3}-(7+2^{\lambda})\eta\right)n.$
\end{center}

 (iii) Since $u\in W_{i_1}^5$, we have $d_{V_{i_1}}(u)\geq 32\eta n$.
Moreover, Applying $\lambda=1$ to (ii) gives $d_{V_{i_1}}(u_j)\geq\big(\frac{1}{3}-9\eta\big)n$ for any $j\in \{1,2\}$.
Combining these with $|V_{i_1}|\leq \big(\frac{1}{3}+\eta\big)n$ due to Lemma \ref{lemma3.005}, we deduce that
\begin{align*}
 \left|N_{V_{i_1}}(u)\cap N_{V_{i_1}}(u_1)\cap N_{V_{i_1}}(u_2)\right|
           &\geq  |N_{V_{i_1}}(u)|+|N_{V_{i_1}}(u_1)|+|N_{V_{i_1}}(u_2)|-2|V_{i_1}|\\
            &> 32\eta n+2\big(\frac{1}{3}-9\eta\big)n-2\big(\frac{1}{3}+\eta\big)n\\
            &\geq  |W^{1}\cup S|+3k+3
\end{align*}
as $\eta<\frac{1}{9(120k+48)}$, $|W^{1}\cup S|\leq 2\eta n$, and $n$ is sufficiently large.
Then, there exists at least $3k+3$ vertices in $V_{i_1}^{1}$ adjacent to $u,u_1$ (and $u_2$).

(iv) Since $u\in  (W_{i_1}^{\lambda}\cup W_{i_3}^{\lambda}) \setminus S$,
we have $u\in  V_{i}\setminus S$ for some $i\in \{i_1,i_3\}$.
By (i), $d_{V_{i_2}}(u)\geq\big(\frac{1}{9}-5\eta\big)n$.
Moreover, by (ii), $d_{V_{i_2}}(u_j)\geq\left(\frac{1}{3}-(7+2^{\lambda})\eta\right)n$ for any $j\in \{1,\dots,s\}$.
Thus,
\begin{align*}
\left|N_{V_{i_2}}(u)\bigcap\left(\bigcap_{j=1}^{s}N_{V_{i_2}}(u_j)\right)\right|
           &\geq  |N_{V_{i_2}}(u)|+\sum_{j=1}^{s}|N_{V_{i_2}}(u_j)|-s|V_{i_2}|\\
            &>  \big(\frac{1}{9}-5\eta\big)n+s\left(\frac{1}{3}-(7+2^{\lambda})\eta\right)n
                       -s\big(\frac{1}{3}+\eta\big)n\\
            &=  \left(\frac{1}{9}-\left((8+2^{\lambda})(3k+1)+5\right)\eta\right)n\\
            &\geq  |W^{\lambda}\cup S|+3k+3,
\end{align*}
as $\eta<\frac{1}{9(120k+48)}$, $|W^{\lambda}\cup S|\leq 2\eta n$, and $n$ is sufficiently large.
Then, there exist at least $3k+3$ vertices in $V_{i_1}^{\lambda}$ adjacent to $u_1,\dots,u_s$ (and $u$).
\end{proof}
In the following two lemmas, we focus on proving $S=\varnothing$.

\begin{lem}\label{lemma3.8}
 For every $i\in \{1,2,3\}$ and $\lambda\in \{1,5\}$, we have $\nu_i^{\lambda}\leq 1$,
where  $\nu_i^{\lambda}=\nu\big(G[V_i^{\lambda}]\big)$.
Moreover, $G[V_i^{\lambda}]$ contains
an independent set $I_i^{\lambda}$ with $|V_i^{\lambda}\setminus I_i^{\lambda}|\leq 2$.
\end{lem}

\begin{proof}
By symmetry, we may only need to prove that $\nu_1^{\lambda}\leq 1$.
Suppose to the contrary $\nu_1^{\lambda}\ge 2$.
Let $u_{1,1}^*u_{1,1},u_{1,2}^{*}u_{1,2}$ be two independent edges in $G[V_1^{\lambda}]$.
Choose a subset $\widehat{V_1}=\{u_{1,1},u_{1,2},\dots,u_{1,k-2}\}\subseteq V_1^{\lambda}$.
By Lemma \ref{lemma3.7} (iv), there exists a subset
  $\widehat{V_2}=\{u_{2,1},u_{2,2},\dots,u_{2,k-2}\}\subseteq V_2^{\lambda}$
such that all vertices in $\widehat{V_2}$ adjacent to all vertices in $\widehat{V_1}\cup \{u_{1,1}^*,u_{1,2}^*\}$.
Again by Lemma \ref{lemma3.7} (iv), there exists a subset
    $\widehat{V_3}=\{u_{3,1},u_{3,2},\dots,u_{3,k-2}\}\subseteq V_3^{\lambda}$
such that all vertices in $\widehat{V_3}$ adjacent to all vertices in
     $\widehat{V_1}\cup \widehat{V_2}\cup\{u_{1,1}^*,u_{1,2}^*\}$.
Set
 $$C^*=u_{1,1}^*u_{1,1}u_{2,1}u_{3,1}u_{1,2}^*u_{1,2}u_{2,2}u_{3,2}\dots u_{1,k}u_{2,k}u_{3,k}u_{1,1}^*.$$
We can see that $(C^*)^2\cong C_{3k+2}^2$ is a subgraph of $G$, contradicting the fact that $G$ is $C_{3k+2}^2$-free.

Therefore, $\nu_i^{\lambda}\leq 1$.
If $\nu_i^{\lambda}=0$, then $V_i^{\lambda}$ is a desired independent set.
If $\nu_i^{\lambda}=1$, then let $u_{1,1}^*u_{1,1}$
be an edges of $G[V_i]$ and $I_i^{\lambda}=V_i^{\lambda}\setminus\{u_{1,1}^*,u_{1,1}\}.$
Now, if $G[I_i^{\lambda}]$ contains an edge,
then $\nu\big(G[V_i^{\lambda}]\big)\geq 2$,
a contradiction.
So, $I_i^{\lambda}$ is an independent set of $G[V_i^{\lambda}]$.
In both cases, we can see that $|V_i^{\lambda}\setminus I_i^{\lambda}|\leq 2$.
\end{proof}

In the following lemma, we shall prove that $S$ is an empty set.
Since $|W^1|\leq\eta^2 n<n$,
we may select a vertex $v^*$ such that
$x_{v^*}=\max\{x_v~|~v\in V(G)\setminus W^1\}$.
We claim that $v^*\notin S$. Indeed, recall that $x_{u^*}=\max\{x_v~|~v\in V(G)\}$. Then
$\rho(G)x_{u^*}\le |W^1|x_{u^*}+(n-|W^1|)x_{v^*}.$
Combining this with $\rho(G)\ge\frac{2n}{3}$ due to Lemma \ref{lemma3.004}, we  obtain that
\begin{equation}\label{align6}
x_{v^*}\ge \frac{\rho(G)-|W^1|}{n-|W^1|}x_{u^*}\ge \frac{\rho(G)-|W^1|}{n}x_{u^*}
    >\frac{2}{3}\big(1-\eta\big)x_{u^*}>\frac{3}{5}x_{u^*},
\end{equation}
where the last inequality holds by $\eta<\frac{1}{9(120k+48)}$.
On the other hand,
\begin{center}
  $\rho(G)x_{v^*}
=\sum\limits_{v\in N_{W^1}(v^*)}x_v+\sum\limits_{v\in N_{G-W^1}(v^*)}x_v
\le |W^1|x_{u^*}+d_G(v^*)x_{v^*}.$
\end{center}
Combining this with $x_{v^*}>\frac{3}{5}x_{u^*}$,
$\rho(G)\ge\frac{2}{3}n$ and $|W^1|\leq  \eta^2n<\frac12\eta n$, we obtain
\begin{center}
 $d_G(v^*)\ge \rho(G)-\frac{x_{u^*}}{x_{v^*}}|W^1|\geq\rho(G)-\frac{5}{3}|W^1|> \Big(\frac{2}{3}-\frac{5}{6}\eta\Big)n.$
\end{center}
Recall that $S=\{v\in V(G)~|~d_G(v)\leq \big(\frac23-6\eta\big)n\}$.
Then $v^*\notin S$, and so $v^*\in V(G)\setminus (W^1\cup S)$.

We may assume without loss of generality that $v^*\in V_{1}^1$.
Then by the definition of $W^1$, we have $|N_{V_1^1}(v^*)|\leq |N_{V_1}(v^*)|<2\eta n$.
Thus,
\begin{align*}
 \rho(G)x_{v^*}&=  \sum_{v\in N_{S\cup W^1}(v^*)}x_v+
                    \sum_{v\in N_{V_1^1}(v^*)}x_v+
                    \sum_{v\in N_{V_2^1\cup V_3^1}(v^*)}x_v \nonumber\\
                    &<  \big(|W^1|x_{u^*}+|S|x_{v^*}\big)+2\eta nx_{v^*}+\sum_{i=2}^{3}\sum_{v\in  V_{i}^1\setminus I_{i}^{1}}x_v+\sum_{v\in I_2^{1}\cup I_3^{1}}x_v\nonumber\\
&\le  \big(|W^1|x_{u^*}+|S|x_{v^*}\big)+3\eta nx_{v^*}+
\sum_{v\in I_2^{1}\cup I_3^{1}}x_v,
\end{align*}
where $I_{i}^{1}$ is an independent set of
$G[V_{i}^1]$
such that $\big|V_{i}^1\setminus I_{i}^{1}\big|\leq2$ (see Lemma \ref{lemma3.8}).
Subsequently,
\begin{align}\label{align7}
\sum_{v\in I_2^{1}\cup I_3^{1}}x_v
>\big(\rho(G)-|S|-3\eta n\big)x_{v^*}-|W^1|x_{u^*}.
\end{align}

\begin{lem}\label{lemma3.9}
$S=\varnothing$.
\end{lem}
\begin{proof}
Suppose to the contrary that there exists a vertex $u_0\in S$.
We may assume without loss of generality that $u_0\in V_1$.
Let $G'$ be the graph obtained from $G$ by deleting edges incident to $u_0$
and joining all possible edges from $I_2^{1}\cup I_3^{1}$ to $u_0$.

We claim that $G'$ is $C_{3k+2}^2$-free.
Otherwise, $G'$ contains a subgraph $H$ isomorphic to $C_{3k+2}^2$.
From the construction of $G'$,
we can see that $u_0\in V(H)$.
Assume that $N_{H}(u_0)=\{u_1,u_2,\dots,u_a\}$,
then $a\leq |V(H)|-1=3k+1$ and $u_1,u_2,\dots,u_a\in I_2^{1}\cup I_3^{1}$ by the definition of $G'$.
Clearly, $I_i^{1}\subseteq V_i^{1}$ for each $i\in \{2,3\}$.
By Lemma \ref{lemma3.7} (iv), we can select a vertex $u\in V_1^1\setminus V(H)$ adjacent to $u_1,u_2,\dots,u_a$.
This implies that $G[(V(H)\setminus \{u_0\})\cup\{u\}]$ contains a copy of $C_{3k+2}^2$, a contradiction.
Therefore, the above claim holds.

In what follows, we shall obtain a contradiction by showing that $\rho(G')>\rho(G)$.
By the definition of $S$, $d_G(u_0)\le (\frac{2}{3}-6\eta)n$.
Combining this with Lemmas \ref{lemma3.004} and \ref{lemma3.006},
we have
\begin{align}\label{align8}
\rho(G)-d_G(u_0)-|S|>5\eta n.
\end{align}
Besides,
\begin{align}\label{align9}
\sum_{v\in N_G(u_0)}x_{v}=\sum_{v\in N_{W^1}(u_0)}x_v+\sum_{v\in N_{G-W^1}(u_0)}x_v
\le |W^1|x_{u^*}+d_G(u_0)x_{v^*}.
\end{align}
Recall that $x_{v^*}>\frac35x_{u^*}$ and $|W^1|\le \eta^2 n$.
Then combining this with \eqref{align7}, \eqref{align8} and \eqref{align9},
we get that
\begin{align*}
    \sum_{v\in I_2\cup I_3}x_{v}-\sum_{v\in N_G(u_0)}x_{v}
    &\ge  \sum_{v\in I_2\cup I_3}x_v-\big(|W^1|x_{u^*}+d_G(u_0)x_{v^*}\big) \nonumber\\
    &>  \big(\rho(G)\!-\!d_G(u_0)\!-\!|S|\!-3\eta n\big)x_{v^*}\!-\!|W^1|x_{u^*}  \nonumber\\
  &> 2\eta n\cdot\frac{3}{5}x_{u^*}-\eta^2 nx_{u^*}>0
\end{align*}
for $n$ sufficiently large. Thus,
\begin{center}
  $\rho(G')-\rho(G) \ge  X^T\big(A(G')-A(G)\big)X
                  = 2x_{u_0}\left(\sum\limits_{v\in I_2\cup I_3}x_{v}-\sum\limits_{v\in N_G(u_0)}x_v\right)>0,$
\end{center}
contradicting the fact that $G$ is extremal with respect to $spex(n,C_{3k+2}^2)$.
Hence, $S=\varnothing$.
\end{proof}

In the following lemmas, we shall give exact characterizations of $W^5$.

\begin{lem}\label{lemma3.10}
$|W^5|\leq 1$.
\end{lem}

\begin{proof}
Suppose to the contrary that $|W^5|\geq 2$.
Then we can divide the proof into the following two cases.

\noindent{\textbf{Case 1.  For some $i\in \{1,2,3\}$, $|W_i^5|\geq 2$.}}

\begin{figure}[hpt]
\centering
\begin{tikzpicture}[x=1.00mm, y=1.00mm, inner xsep=0pt, inner ysep=0pt, outer xsep=0pt, outer ysep=0pt]
\path[line width=0mm] (41.09,82.77) rectangle +(73.84,45.99);
\definecolor{L}{rgb}{0,0,0}
\path[line width=0.30mm, draw=L] (70.02,120.23);
\path[line width=0.30mm, draw=L] (75.05,120.06);
\path[line width=0.30mm, draw=L] (80.00,120.00) -- (90.00,110.00);
\path[line width=0.30mm, draw=L] (90.00,110.00) -- (90.00,100.00);
\path[line width=0.30mm, draw=L] (70.00,90.00) -- (80.00,90.00);
\path[line width=0.30mm, draw=L] (80.00,90.00) -- (90.00,100.00);
\draw(92.88,109.47) node[anchor=base west]{\fontsize{14.23}{17.07}\selectfont $u_{2,1}$};
\draw(92.88,99.47) node[anchor=base west]{\fontsize{14.23}{17.07}\selectfont $u_{3,1}$};
\draw(80.16,85.55) node[anchor=base west]{\fontsize{14.23}{17.07}\selectfont $u_{1,2}^*$};
\draw(65.94,122.84) node[anchor=base west]{\fontsize{14.23}{17.07}\selectfont $u_{1,1}^*$};
\path[line width=0.30mm, draw=L] (80.00,120.00) -- (90.00,100.00);
\path[line width=0.30mm, draw=L] (90.00,110.00) -- (80.00,90.00);
\path[line width=0.30mm, draw=L] (70.00,120.00) -- (90.00,110.00);
\definecolor{L}{rgb}{0,0,1}
\definecolor{F}{rgb}{0,0,1}
\path[line width=0.30mm, draw=L, fill=F] (89.62,110.00) circle (1.00mm);
\definecolor{L}{rgb}{0,0,0}
\definecolor{F}{rgb}{0,0,0}
\path[line width=0.30mm, draw=L, fill=F] (90.00,100.00) circle (1.00mm);
\path[line width=0.30mm, draw=L] (75.57,120.58);
\path[line width=0.30mm, draw=L] (89.32,99.97);
\draw(78.91,122.49) node[anchor=base west]{\fontsize{14.23}{17.07}\selectfont $u_{1,1}$};
\draw(66.08,85.55) node[anchor=base west]{\fontsize{14.23}{17.07}\selectfont $u_{1,2}$};
\path[line width=0.30mm, draw=L] (70.00,120.00) -- (80.00,120.00);
\path[line width=0.30mm, draw=L] (60.00,100.00) -- (70.00,90.00);
\path[line width=0.30mm, draw=L] (60.00,110.00) -- (60.00,100.00);
\path[line width=0.30mm, draw=L] (60.00,110.00) -- (70.00,120.00);
\path[line width=0.30mm, draw=L, fill=F] (60.00,110.00) circle (1.00mm);
\draw(50.68,99.47) node[anchor=base west]{\fontsize{14.23}{17.07}\selectfont $u_{2,2}$};
\draw(50.68,109.47) node[anchor=base west]{\fontsize{14.23}{17.07}\selectfont $u_{3,2}$};
\path[line width=0.30mm, draw=L] (70.00,90.00) -- (90.00,100.00);
\path[line width=0.30mm, draw=L] (80.00,90.00) -- (60.00,100.00);
\path[line width=0.30mm, draw=L] (60.00,110.00) -- (70.00,90.00);
\path[line width=0.30mm, draw=L] (60.00,110.00) -- (80.00,120.00);
\path[line width=0.30mm, draw=L] (70.00,120.00) -- (60.00,100.00);
\definecolor{L}{rgb}{1,0,0}
\definecolor{F}{rgb}{1,0,0}
\path[line width=0.30mm, draw=L, fill=F] (80.00,90.00) circle (1.00mm);
\path[line width=0.30mm, draw=L, fill=F] (70.00,90.00) circle (1.00mm);
\definecolor{L}{rgb}{0,0,1}
\definecolor{F}{rgb}{0,0,1}
\path[line width=0.30mm, draw=L, fill=F] (60.00,100.00) circle (1.00mm);
\definecolor{L}{rgb}{1,0,0}
\definecolor{F}{rgb}{1,0,0}
\path[line width=0.30mm, draw=L, fill=F] (70.00,120.00) circle (1.00mm);
\path[line width=0.30mm, draw=L, fill=F] (80.00,120.00) circle (1.00mm);
\end{tikzpicture}
\caption{The subgraph $H_1$ in $G$.}{\label{figure.2}}
\end{figure}
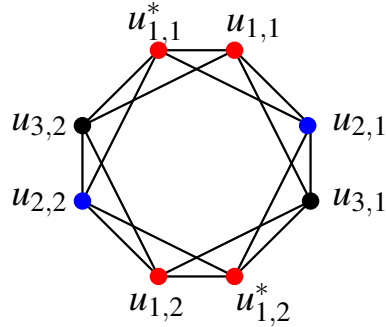

We may assume without loss of generality that $|W_1^5|\geq 2$ and $u_{1,1}^*,u_{1,2}^*\in W_1^5$.
Recall that $d_G(u)> \big(\frac{2}{3}-6\eta\big)n$ for any $u\in V(G)$ as $S=\varnothing$.
Then,
\begin{center}
  $|N_G(u_{1,1}^*)\cap N_G(u_{1,2}^*)|\geq 2\big(\frac{2}{3}-6\eta\big)n-n=\big(\frac{1}{3}-12\eta\big)n.$
\end{center}
We first claim that $|N_{V_i}(u_{1,1}^*)\cap N_{V_i}(u_{1,2}^*)|\geq 2\eta n$ for some $i\in \{2,3\}$.
Otherwise, $|N_{V_i}(u_{1,1}^*)\cap N_{V_i}(u_{1,2}^*)|< 2\eta n$ for each $i\in \{2,3\}$, which implies that
\begin{center}
  $|N_{V_1}(u_{1,1}^*)\cap N_{V_1}(u_{1,2}^*)|
      =|N_G(u_{1,1}^*)\cap N_G(u_{1,2}^*)|-\sum\limits_{i=2}^{3}|N_{V_i}(u_{1,1}^*)\cap N_{V_i}(u_{1,2}^*)|
      \geq \big(\frac{1}{3}-16\eta\big)n.$
\end{center}
Recall that $V(G)=V_1\cup V_2\cup V_3$ is a partition of $G$ such that $\sum_{1\leq i<j\leq 3}e(V_i,V_j)$ attains the maximum. Then for any $u\in \{u_{1,1}^*,u_{1,2}^*\}$, we have
\begin{center}
 $|N_{V_2}(u)|\geq |N_{V_1}(u)|\geq |N_{V_1}(u_{1,1}^*)\cap N_{V_1}(u_{1,2}^*)| \geq \big(\frac{1}{3}-16\eta\big)n.$
\end{center}
 Combining this with $|V_2|\le (\frac{1}{3}+\eta)n$ due to Lemma \ref{lemma3.005}, we obtain that
\begin{center}
  $|N_{V_2}(u_{1,1}^*)\cap N_{V_2}(u_{1,2}^*)|
     \geq 2\big(\frac{1}{3}-16\eta\big)n-\big(\frac{1}{3}+\eta\big)n
     =\big(\frac{1}{3}-33\eta\big)n>2\eta n,$
\end{center}
a contradiction.
Hence, the claim holds.

Now we may assume without loss of generality that $|N_{V_2}(u_{1,1}^*)\cap N_{V_2}(u_{1,2}^*)|\geq 2\eta n$.
Recall that $|W_2^1|\leq \eta^2 n$.
Thus,  there exist two vertices $u_{2,1},u_{2,2}\in (N_{V_2}(u_{1,1}^*)\cap N_{V_2}(u_{1,2}^*))\setminus W^1$.
Moreover, since $S=\varnothing$, we have $u_{2,1},u_{2,2}\in V_2^1$.
Recall that $u_{1,1}^*,u_{1,2}^*\in W_1^5$.
By Lemma \ref{lemma3.7} (iii), there exist a vertex $u_{1,1}$ in $V_1^1$ adjacent to $u_{1,1}^*$ and $u_{2,1}$, and a vertex $u_{1,2}$ ($u_{1,2}\neq u_{1,1}$) in $V_1^1$ adjacent to $u_{1,2}^*$ and
$u_{2,2}$ (see Figure \ref{figure.2}).
By Lemma \ref{lemma3.7} (iv), there exist a vertex $u_{3,1}\in V_3^1$ adjacent to $u_{1,1},u_{2,1},u_{1,2}^*,u_{1,2}$,
and a vertex $u_{3,2}\in V_3^1$ ($u_{3,2}\neq u_{3,1}$) adjacent to $u_{1,2},u_{2,2},u_{1,1}^*,u_{1,1}$.
If $k=2$, then $G$ contains an $H_1$ isomorphic to $ C_{8}^2$ (see Figure \ref{figure.2}), a contradiction.
We now assume that $k\geq 3$.
By Lemma \ref{lemma3.7} (iv), we can recursively select subsets $\widehat{V_i}=\{v_{i,1},v_{i,2},\dots,v_{i,k-2}\}$ in $V_i^1\setminus V(H_1)$
for each $i\in \{1,2,3\}$ such that all vertices in $\widehat{V_1}$ adjacent to $u_{2,2}$ and $u_{3,2}$;
all vertices in $\widehat{V_2}$ adjacent to all vertices in $\widehat{V_1}\cup \{u_{1,1}^*,u_{1,1},u_{3,2}\}$;
all vertices in $\widehat{V_3}$ adjacent to all vertices in $\widehat{V_1}\cup \widehat{V_2}\cup\{u_{1,1}^*,u_{1,1},u_{2,2}\}$.
Set $$C^1=u_{1,1}^*u_{1,1}u_{2,1}u_{3,1}u_{1,2}^*u_{1,2}u_{2,2}u_{3,2}v_{1,1}v_{2,1}v_{3,1}\dots v_{1,k-2}v_{2,k-2}v_{3,k-2}u_{1,1}^*.$$
We can see that $(C^1)^2\cong C_{3k+2}^2$ is a subgraph of $G$, contradicting the fact that $G$ is $C_{3k+2}^2$-free.

\noindent{\textbf{Case 2. For each $i\in \{1,2,3\}$, $|W_i^5|\leq 1$.}}

Together with $|W^5|\geq 2$, we may assume that $u_{1,1}^*\in W_1^5$ and $u_{3,1}^*\in W_3^5$.
Now we consider the following two subcases.

\noindent{\textbf{Subcase 2.1. $|N_{V_2}(u_{1,1}^*)\cap N_{V_2}(u_{3,1}^*)|\geq 20\eta n$.}}

Then, there exist vertices $u_{2,1}$ and $u_{2,2}$ in $V_2^1$ adjacent to $u_{1,1}^*$ and $u_{3,1}^*$ since $|W_2^1|<|W^1|<\eta^2n$.
Using Lemma \ref{lemma3.7} (iii) recursively, there exists a vertex $u_{1,1}\in V_1^1$ adjacent to $u_{1,1}^*,u_{2,1}$;
and a vertex $u_{3,1}\in V_3^1$ adjacent to $u_{1,1},u_{2,1},u_{3,1}^*$.
By Lemma \ref{lemma3.7} (iv), we can recursively select a sequence of vertices $u_{1,2},u_{3,2}$
such that the vertex $u_{1,2}\in V_1^1$ is adjacent to $u_{3,1},u_{3,1}^*,u_{2,2}$;
the vertex $u_{3,2}\in V_3^1$ is adjacent to $u_{1,2},u_{2,2},u_{1,1}^*,u_{1,1}$ (see Figure \ref{figure.3}).
If $k=2$, then $G$ contains an $H_2$ isomorphic to $C_{8}^2$, a contradiction.
Assume now that $k\geq 3$.
We give a similar definition of $\widehat{V_i}$ in $V_i^1\setminus V(H_2)$ for each $i\in \{1,2,3\}$ as Case 1 and set
 $$C^2=u_{1,1}^*u_{1,1}u_{2,1}u_{3,1}u_{3,1}^*u_{1,2}u_{2,2}u_{3,2}v_{1,1}v_{2,1}v_{3,1}\dots v_{1,k-2}v_{2,k-2}v_{3,k-2}u_{1,1}^*.$$
We can see that $(C^2)^2\cong C_{3k+2}^2$ is a subgraph of $G$, contradicting the fact that $G$ is $C_{3k+2}^2$-free.

\begin{figure}[hpt]
\centering
\begin{tikzpicture}[x=1.00mm, y=1.00mm, inner xsep=0pt, inner ysep=0pt, outer xsep=0pt, outer ysep=0pt]
\path[line width=0mm] (41.09,82.77) rectangle +(78.69,45.99);
\definecolor{L}{rgb}{0,0,0}
\path[line width=0.30mm, draw=L] (70.02,120.23);
\path[line width=0.30mm, draw=L] (75.05,120.06);
\path[line width=0.30mm, draw=L] (80.00,120.00) -- (90.00,110.00);
\path[line width=0.30mm, draw=L] (90.00,110.00) -- (90.00,100.00);
\path[line width=0.30mm, draw=L] (70.00,90.00) -- (80.00,90.00);
\path[line width=0.30mm, draw=L] (80.00,90.00) -- (90.00,100.00);
\draw(91.98,109.47) node[anchor=base west]{\fontsize{14.23}{17.07}\selectfont $u_{2,1}$};
\draw(91.88,98.02) node[anchor=base west]{\fontsize{14.23}{17.07}\selectfont $u_{3,1}$};
\draw(80.16,85.55) node[anchor=base west]{\fontsize{14.23}{17.07}\selectfont $u_{3,1}^*$};
\draw(65.94,122.84) node[anchor=base west]{\fontsize{14.23}{17.07}\selectfont $u_{1,1}^*$};
\path[line width=0.30mm, draw=L] (80.00,120.00) -- (90.00,100.00);
\path[line width=0.30mm, draw=L] (90.00,110.00) -- (80.00,90.00);
\path[line width=0.30mm, draw=L] (70.00,120.00) -- (90.00,110.00);
\definecolor{L}{rgb}{0,0,1}
\definecolor{F}{rgb}{0,0,1}
\path[line width=0.30mm, draw=L, fill=F] (89.62,110.00) circle (1.00mm);
\definecolor{L}{rgb}{0,0,0}
\definecolor{F}{rgb}{0,0,0}
\path[line width=0.30mm, draw=L, fill=F] (90.00,100.00) circle (1.00mm);
\path[line width=0.30mm, draw=L] (75.57,120.58);
\path[line width=0.30mm, draw=L] (89.32,99.97);
\draw(78.91,122.49) node[anchor=base west]{\fontsize{14.23}{17.07}\selectfont $u_{1,1}$};
\draw(66.08,85.85) node[anchor=base west]{\fontsize{14.23}{17.07}\selectfont $u_{1,2}$};
\path[line width=0.30mm, draw=L] (70.00,120.00) -- (80.00,120.00);
\path[line width=0.30mm, draw=L] (60.00,100.00) -- (70.00,90.00);
\path[line width=0.30mm, draw=L] (60.00,110.00) -- (60.00,100.00);
\path[line width=0.30mm, draw=L] (60.00,110.00) -- (70.00,120.00);
\path[line width=0.30mm, draw=L, fill=F] (60.00,110.00) circle (1.00mm);
\draw(50.68,99.47) node[anchor=base west]{\fontsize{14.23}{17.07}\selectfont $u_{2,2}$};
\draw(50.68,109.47) node[anchor=base west]{\fontsize{14.23}{17.07}\selectfont $u_{3,2}$};
\path[line width=0.30mm, draw=L] (70.00,90.00) -- (90.00,100.00);
\path[line width=0.30mm, draw=L] (80.00,90.00) -- (60.00,100.00);
\path[line width=0.30mm, draw=L] (60.00,110.00) -- (70.00,90.00);
\path[line width=0.30mm, draw=L] (60.00,110.00) -- (80.00,120.00);
\path[line width=0.30mm, draw=L] (70.00,120.00) -- (60.00,100.00);
\path[line width=0.30mm, draw=L, fill=F] (80.00,90.00) circle (1.00mm);
\definecolor{L}{rgb}{1,0,0}
\definecolor{F}{rgb}{1,0,0}
\path[line width=0.30mm, draw=L, fill=F] (70.00,90.00) circle (1.00mm);
\definecolor{L}{rgb}{0,0,1}
\definecolor{F}{rgb}{0,0,1}
\path[line width=0.30mm, draw=L, fill=F] (60.00,100.00) circle (1.00mm);
\definecolor{L}{rgb}{1,0,0}
\definecolor{F}{rgb}{1,0,0}
\path[line width=0.30mm, draw=L, fill=F] (70.00,120.00) circle (1.00mm);
\path[line width=0.30mm, draw=L, fill=F] (80.00,120.00) circle (1.00mm);
\end{tikzpicture}
\caption{The subgraph $H_2$ in $G$.}{\label{figure.3}}
\end{figure}
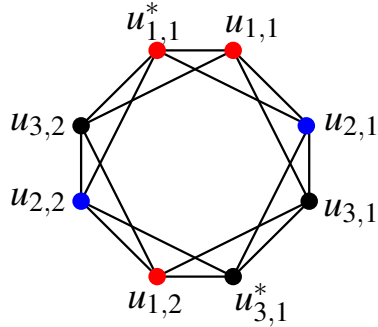

\noindent{\textbf{Subcase 2.2. $|N_{V_2}(u_{1,1}^*)\cap N_{V_2}(u_{3,1}^*)|< 20\eta n$.}}

For any integer $i\in \{1,3\}$, we first claim that $|N_{V_i}(u_{1,1}^*)\cap N_{V_i}(u_{3,1}^*)|\geq 20\eta n$.
By symmetry, it suffices to prove that $|N_{V_1}(u_{1,1}^*)\cap N_{V_1}(u_{3,1}^*)|\geq 20\eta n$.
Suppose to the contrary that $|N_{V_1}(u_{1,1}^*)\cap N_{V_1}(u_{3,1}^*)|< 20\eta n$.
Since $S$ is empty, $d_G(u)> \big(\frac{2}{3}-6\eta\big)n$ for any $u\in V(G)$.
Then, $|N_G(u_{1,1}^*)\cap N_G(u_{3,1}^*)|\geq 2\big(\frac{2}{3}-6\eta\big)n-n=\big(\frac{1}{3}-12\eta\big)n$.
Thus,
\begin{center}
 $|N_{V_3}(u_{1,1}^*)\cap N_{V_3}(u_{3,1}^*)|
      =|N_G(u_{1,1}^*)\cap N_G(u_{3,1}^*)|-\sum\limits_{i=1}^{2}|N_{V_i}(u_{1,1}^*)\cap N_{V_i}(u_{1,2}^*)|
      \geq \big(\frac{1}{3}-52\eta\big)n.$
\end{center}
Recall that  $V(G)=V_1\cup V_2\cup V_3$ is a partition of $G$ such that $\sum_{1\leq i<j\leq 3}e(V_i,V_j)$ attains the maximum. Then
\begin{center}
$|N_{V_2}(u_{3,1}^*)|\geq |N_{V_3}(u_{3,1}^*)|\geq |N_{V_3}(u_{3,1}^*)\cap N_{V_3}(u_{1,1}^*)| \geq \big(\frac{1}{3}-52\eta\big)n.$
\end{center}
By Lemma \ref{lemma3.7} (i),  $|N_{V_2}(u_{1,1}^*)|\geq\big(\frac{1}{9}-5\eta\big)n$.
Thus,
\begin{align*}
  |N_{V_2}(u_{1,1}^*)\cap N_{V_2}(u_{3,1}^*)| &\ge |N_{V_2}(u_{1,1}^*)|+ |N_{V_2}(u_{3,1}^*)|-|V_2|\\
       &  \geq \big(\frac{1}{9}-5\eta\big)n+\big(\frac{1}{3}-52\eta\big)n-\big(\frac{1}{3}+\eta\big)n \\
    & =\big(\frac{1}{9}-58\eta\big)n>20\eta n,
\end{align*}
a contradiction. Hence the claim holds.
Since $|N_{V_3}(u_{1,1}^*)\cap N_{V_3}(u_{3,1}^*)|\geq 20\eta n$, there exists a vertex $u_{3,1}\in V_3^1$ adjacent to $u_{3,1}^*$ and $u_{1,1}^*$.
Moreover, since $|N_{V_1}(u_{1,1}^*)\cap N_{V_1}(u_{3,1}^*)|\geq 20\eta n$ and
      $d_{V_1}(u_{3,1})\geq \big(\frac{1}{3}-9\eta\big)n$ by Lemma \ref{lemma3.7} (ii), we have
      \begin{center}
        $|N_{V_1}(u_{1,1}^*)\cap N_{V_1}(u_{3,1}^*)\cap N_{V_1}(u_{3,1})|
      \geq 20\eta n+\big(\frac{1}{3}-9\eta\big)n-\big(\frac{1}{3}+\eta\big)n
      \geq 10\eta n.$
      \end{center}
Hence, there exists a vertex $u_{1,1}\in V_1^1$ adjacent to $u_{3,1}^*,u_{3,1}$ and $u_{1,1}^*$.

\begin{figure}[hpt]
\centering
\begin{tikzpicture}[x=1.00mm, y=1.00mm, inner xsep=0pt, inner ysep=0pt, outer xsep=0pt, outer ysep=0pt]
\path[line width=0mm] (41.09,82.77) rectangle +(78.69,45.99);
\definecolor{L}{rgb}{0,0,0}
\path[line width=0.30mm, draw=L] (70.02,120.23);
\path[line width=0.30mm, draw=L] (75.05,120.06);
\path[line width=0.30mm, draw=L] (80.00,120.00) -- (90.00,110.00);
\path[line width=0.30mm, draw=L] (90.00,110.00) -- (90.00,100.00);
\path[line width=0.30mm, draw=L] (70.00,90.00) -- (80.00,90.00);
\path[line width=0.30mm, draw=L] (80.00,90.00) -- (90.00,100.00);
\draw(91.98,109.47) node[anchor=base west]{\fontsize{14.23}{17.07}\selectfont $u_{1,1}$};
\draw(91.88,98.02) node[anchor=base west]{\fontsize{14.23}{17.07}\selectfont $u_{1,1}^*$};
\draw(80.16,85.55) node[anchor=base west]{\fontsize{14.23}{17.07}\selectfont $u_{2,1}$};
\draw(65.94,122.84) node[anchor=base west]{\fontsize{14.23}{17.07}\selectfont $u_{3,1}^*$};
\path[line width=0.30mm, draw=L] (80.00,120.00) -- (90.00,100.00);
\path[line width=0.30mm, draw=L] (90.00,110.00) -- (80.00,90.00);
\path[line width=0.30mm, draw=L] (70.00,120.00) -- (90.00,110.00);
\definecolor{L}{rgb}{0,0,1}
\definecolor{F}{rgb}{0,0,1}
\path[line width=0.30mm, draw=L, fill=F] (89.62,110.00) circle (1.00mm);
\definecolor{L}{rgb}{0,0,0}
\path[line width=0.30mm, draw=L, fill=F] (90.00,100.00) circle (1.00mm);
\path[line width=0.30mm, draw=L] (75.57,120.58);
\path[line width=0.30mm, draw=L] (89.32,99.97);
\draw(78.91,122.49) node[anchor=base west]{\fontsize{14.23}{17.07}\selectfont $u_{3,1}$};
\draw(50.00,109.47) node[anchor=base west]{\fontsize{14.23}{17.07}\selectfont $v_{2,1}$};
\path[line width=0.30mm, draw=L] (70.00,120.00) -- (80.00,120.00);
\path[line width=0.30mm, draw=L] (60.00,100.00) -- (70.00,90.00);
\path[line width=0.30mm, draw=L] (60.00,110.00) -- (60.00,100.00);
\path[line width=0.30mm, draw=L] (60.00,110.00) -- (70.00,120.00);
\definecolor{F}{rgb}{0,0,0}
\path[line width=0.30mm, draw=L, fill=F] (60.00,110.00) circle (1.00mm);
\draw(50.00,99.47) node[anchor=base west]{\fontsize{14.23}{17.07}\selectfont $v_{1,1}$};
\draw(66.08,85.85) node[anchor=base west]{\fontsize{14.23}{17.07}\selectfont $v_{3,1}$};
\path[line width=0.30mm, draw=L] (70.00,90.00) -- (90.00,100.00);
\path[line width=0.30mm, draw=L] (80.00,90.00) -- (60.00,100.00);
\path[line width=0.30mm, draw=L] (60.00,110.00) -- (70.00,90.00);
\path[line width=0.30mm, draw=L] (60.00,110.00) -- (80.00,120.00);
\path[line width=0.30mm, draw=L] (70.00,120.00) -- (60.00,100.00);
\path[line width=0.30mm, draw=L, fill=F] (80.00,90.00) circle (1.00mm);
\definecolor{L}{rgb}{1,0,0}
\definecolor{F}{rgb}{1,0,0}
\path[line width=0.30mm, draw=L, fill=F] (70.00,90.00) circle (1.00mm);
\definecolor{L}{rgb}{0,0,1}
\definecolor{F}{rgb}{0,0,1}
\path[line width=0.30mm, draw=L, fill=F] (60.00,100.00) circle (1.00mm);
\definecolor{L}{rgb}{1,0,0}
\definecolor{F}{rgb}{1,0,0}
\path[line width=0.30mm, draw=L, fill=F] (70.00,120.00) circle (1.00mm);
\path[line width=0.30mm, draw=L, fill=F] (80.00,120.00) circle (1.00mm);
\end{tikzpicture}
\caption{The subgraph $(C^3)^2$ of $G$ for case $k=2$.}{\label{figure.4}}
\end{figure}
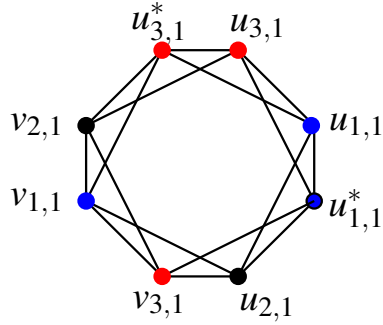

By Lemma \ref{lemma3.7} (iv), there exists a vertex $u_{2,1}\in V_2^1$ adjacent to $u_{1,1}$ and $u_{1,1}^*$.
Using Lemma \ref{lemma3.7} (iv) again, we can recursively select $\widehat{V_i}=\{v_{i,1},v_{i,2},\dots,v_{i,k-1}\}\subseteq V_i^1\setminus \{u_{3,1}^*,u_{3,1},u_{1,1},u_{1,1}^*,u_{2,1}\}$
for each $i\in \{1,2,3\}$ such that all vertices in $\widehat{V_1}$ adjacent to $u_{2,1}$ and $u_{3,1}^*$;
all vertices in $\widehat{V_2}$ adjacent to all vertices in $\widehat{V_1}\cup \{u_{3,1}^*,u_{3,1}\}$;
all vertices in $\widehat{V_3}$ adjacent to all vertices in $\widehat{V_1}\cup \widehat{V_2}\cup\{u_{1,1}^*,u_{2,1}\}$.
Set
 $$C^3=u_{3,1}^*u_{3,1}u_{1,1}u_{1,1}^*u_{2,1}v_{3,1}v_{1,1}v_{2,1}\dots v_{3,k-1}v_{1,k-1}v_{2,k-1}u_{3,1}^*.$$
We can see that $(C^3)^2\cong C_{3k+2}^2$ is a subgraph of $G$ (see $(C^3)^2$ for $k=2$ in Figure \ref{figure.4}),
contradicting the fact that $G$ is $C_{3k+2}^2$-free.
\end{proof}

The following lemmas are used to give a lower bound for
entries of vertices of $G$.

\begin{lem}\label{lemma3.11}
 For each $v\in V(G)$, we have $x_v>\frac{3}{5}x_{u^*}$.
\end{lem}

\begin{proof}
Recall that $\rho(G)>\frac{2n}{3}$ and
$|W^1|\leq\eta^2 n$.
Then $\eta n<\frac32 \eta \rho(G)$ and so $|W^1|<\frac32\eta^2\rho(G)$.
Moreover, $S=\varnothing$ by Lemma \ref{lemma3.8}.
Combining these with \eqref{align6} and (\ref{align7}), we obtain that
\begin{center}
$\sum\limits_{v\in I_2^{1}\cup I_3^{1}}x_v
>\big(\rho(G)-3\eta n\big)x_{v^*}-\frac32\eta^2 \rho(G)x_{u^*}>\big(1-5\eta)\rho(G)x_{v^*}. $
\end{center}
Again, from \eqref{align6} we know that $x_{v^*}>\frac{2}{3}\big(1-\eta\big)x_{u^*}.$
Thus, for $n$ sufficiently large,
\begin{center}
 $\sum\limits_{v\in I_2^{1}\cup I_3^{1}}x_v>
\Big(\frac{2}{3}-4\eta\Big)\rho(G)x_{u^*}.$
\end{center}

Now suppose to the contrary, we may assume that there exists $u_0\in V_1$ such that $x_{u_0}\leq\frac{3}{5}x_{u^*}$.
Let $G'$ be the graph obtained from $G$ by deleting edges incident to $u_0$
and joining all edges from $I_2^{1}\cup I_3^{1}$ to $u_0$.
By a similar discussion as in the proof of Lemma \ref{lemma3.8},
we get that $G'$ is $C_{3k+2}^2$-free.
However,
\begin{center}
  $\sum\limits_{v\in I_2^{1}\cup I_3^{1}}x_{v}-\sum\limits_{v\in N_G(u_0)}x_{v}=\sum\limits_{v\in I_2^{1}\cup I_3^{1}}x_v -\rho(G)x_{u_0}>\left(\frac{2}{3}-4\eta-\frac35\right)\rho(G)x_{u^*}>0,$
\end{center}
which leads to that
\begin{center}
  $\rho(G')-\rho(G) \ge  X^T\big(A(G')-A(G)\big)X
                  = 2x_{u_0}\Big(\sum\limits_{v\in I_2^{1}\cup I_3^{1}}x_v-\sum\limits_{v\in N_G(u_0)}x_v\Big)>0,$
\end{center}
contradicting the fact that $G$ is extremal with respect to $spex(n,C_{3k+2}^2)$.
\end{proof}

\begin{lem}\label{lemma3.12}
$|W^5|=1$ and $\sum\limits_{i=1}^{3}\nu_i=1$, where $\nu_i=\nu(G[V_i])$.
\end{lem}

\begin{proof}
We first claim that $|W^5|=1$.
Otherwise, $|W^5|=0$ by Lemma \ref{lemma3.10}.
For every $i\in \{1,2,3\}$, by Lemma \ref{lemma3.8},
we have $\nu_i^5\leq 1$, and hence
    $$\nu_i=\nu(G[V_i])=\nu(G[V_i\setminus (W^{5}\cup S)])=\nu_i^5\leq 1.$$
It is not hard to check that $|W_i^1|\leq \nu_i\leq 1$.
For any vertex $u\in W_i^1$, we can see that $u\notin W_i^5$ as $W_i^5=\varnothing$, which implies that $d_{V_i}(u)\leq 32 \eta n$.
Then we can further obtain that
\begin{align}\label{align20}
e(G[V_i])\leq v_i^1+|W_i^1|\cdot 32 \eta n\leq 33\eta n.
\end{align}

Let $G'$ be the graph obtained from $G$ by deleting all edges in $\bigcup_{i=1}^{3}E(G[V_i])$ and adding all possible edges from $u_0\in V_1$ to all vertices in $V_1\setminus\{u_0\}$.
Then, $G'$ is a subgraph of $K_1+K_{|V_1|-1,|V_2|,|V_3|}$, which implies that $G'$ is $C_{3k+2}^2$-free.
On the other hand, by Lemmas \ref{lemma3.11} and \ref{lemma3.005} and \eqref{align20},
\begin{align*}
\rho(G')-\rho(G)
    &\geq  {X^T(A(G')-A(G))X} \nonumber\\
    &\geq  2\left(\sum_{v\in V_1}x_{u_0}x_{v}-\sum_{vw\in \bigcup_{i=1}^{3}E(G[V_i])}x_vx_w\right) \nonumber\\
  &\geq 2\left(\frac{9}{25}\big(\frac{1}{3}n-\eta n-1\big)-99\eta n\right)x_{u^*}^2>0.
\end{align*}
This contradicts that $G$ is $C_{3k+2}^2$-free.

Since $|W^5|=1$, we may assume that $W_1^5=\{u_{1,1}^*\}$. Then $\nu_1\ge 1$.
We first claim that $\nu_1= 1$.
Suppose to the contrary, and let $u_{1,1}^*u_{1,1},u_{1,2}^*u_{1,2}$ be two independent edges in $G[V_1]$.
One can observe that $u_{1,1},u_{1,2}^*,u_{1,2}\in V_1^5$.
Choose a subset $\widehat{V_1}=\{u_{1,1},u_{1,2},\dots,u_{1,k}\}\subseteq V_1^5\setminus \{u_{1,2}^*\}$.
Setting $\lambda=5$ in Lemma \ref{lemma3.7} (iv), we can recursively select subsets $\widehat{V_i}=\{u_{i,1},u_{i,2},\dots,u_{i,k}\}$ in
$V_i^5$ for each $i\in \{2,3\}$ such that
all vertices in $\widehat{V_2}$ adjacent to all vertices in $\widehat{V_1}\cup \{u_{1,1}^*,u_{1,2}^*\}$;
all vertices in $\widehat{V_3}$ adjacent to all vertices in $\widehat{V_1}\cup \widehat{V_2}\cup\{u_{1,1}^*,u_{1,2}^*\}$.
Set
 $$C^4=u_{1,1}u_{1,1}^*u_{2,1}u_{3,1}u_{1,2}u_{1,2}^*u_{2,2}u_{3,2}\dots u_{1,k}u_{2,k}u_{3,k}u_{1,1}.$$
We can see that $(C^4)^2\cong C_{3k+2}^2$ is a subgraph of $G$, contradicting the fact that $G$ is $C_{3k+2}^2$-free.

Now we shall prove that $\sum_{i=1}^{3}\nu_i=1$.
Otherwise, either $\nu_2= 1$ or $\nu_3= 1$.
We may assume that $\nu_3=1$, and $u_{1,1}^*u_{1,1},u_{3,1}^*u_{3,1}$ are two independent edges in $G[V_1]$ and $G[V_3]$, respectively.
Setting $\lambda=5$ in Lemma \ref{lemma3.7} (iv),
we can recursively select subsets $\widehat{V_i}=\{u_{i,1},u_{i,2},\dots,u_{i,k}\}$ in $V_i^5\setminus \{u_{1,1}^*,u_{3,1}^*\}$
for each $i\in \{1,2,3\}$ such that
all vertices in $\widehat{V_1}$ adjacent to all vertices in $\{u_{3,1}^*,u_{3,1}\}$;
all vertices in $\widehat{V_2}$ adjacent to all vertices in $\widehat{V_1}\cup \{u_{1,1}^*,u_{3,1}^*,u_{3,1}\}$;
all vertices in $\widehat{V_3}$ adjacent to all vertices in $\widehat{V_1}\cup \widehat{V_2}\cup\{u_{1,1}^*\}$.
Set
 $$C^5=u_{1,1}^*u_{1,1}u_{2,1}u_{3,1}^*u_{3,1}u_{1,2}u_{2,2}u_{3,2}\dots u_{1,k}u_{2,k}u_{3,k}u_{1,1}^*.$$
We can see that $(C^5)^2\cong C_{3k+2}^2$ is a subgraph of $G$, contradicting the fact that $G$ is $C_{3k+2}^2$-free.
\end{proof}

Now we complete the proof of Theorem \ref{theorem1.3} as follows.

\begin{proof}
Recall that $G(n)=K_{1}+T_{n-1,3}$ and we shall prove $G\cong G(n)$.
By Lemma \ref{lemma3.12}, we may assume that $W^5=\{u^*\}$.
Furthermore, We shall show that $d_G(u^*)=n-1$.
Suppose to the contrary, then we can find a non-neighbor $v$ of $u^*$ in $G$.
Let $G'=G+\{u^*v\}$. Then $\rho(G')>\rho(G)$.
On the other hand,
by Lemma \ref{lemma3.12} we can find that $G'$ is a subgraph of $K_1+K_3(n_1,n_2,n_3)$, where $|V_i\setminus W^5|=n_i$ for $i\in \{1,2,3\}$.
This indicates that $G'$ is $C_{3k+2}^2$-free,
contradicting the fact that $G$ is extremal with respect to $spex(n,C_{3k+2}^2)$.

Assume without loss of generality that $n_1\geq n_2\geq n_3$.
By Lemma \ref{lemma3.12}, we can see that $G-W^5\subseteq K_{n_1,n_2,n_3}$.
Since $G$ is extremal with respect to $spex(n,C_{3k+2}^2)$, we have $G-W^5\cong K_3(n_1,n_2,n_3)$.
To prove $G\cong G(n)$,
it suffices to show $G-W^5\cong T_{n-1,3}$, that is, $n_1-n_3\le 1$.
Suppose to the contrary that $n_1\geq n_3+2$.
By symmetry, we may assume $x_u=x_i$ for each $u\in V_i\setminus W^5$
and $i\in\{1,2,3\}$. Thus,
\begin{align*}
\rho(G)x_1=n_2x_2+n_3x_3+x_{u^*},
~\rho(G)x_2=n_1x_1+n_3x_3+x_{u^*},
~\rho(G)x_3=n_1x_1+n_2x_2+x_{u^*},
\end{align*}
and $\rho(G)x_{u^*}=n_1x_1+n_2x_2+n_3x_3.$
It follows that
\begin{align}\label{align11}
x_1=\frac{\rho(G)+1}{\rho(G)+n_1}x_{u^*}~~\text{and}
~~x_3=\frac{\rho(G)+1}{\rho(G)+n_3}x_{u^*}.
\end{align}
Choose a vertex $u_0\in V_1\setminus W^5$.
Let $G''$ be the graph obtained from $G$ by deleting edges
from $u_0$ to $V_3\setminus W^5$
and adding all edges from $u_0$ to $V_1\setminus (W^5\cup\{u_0\})$.
Then $G''\cong K_1+K_{n_1-1,n_2,n_3+1}$, and thus $G''$ is still $C_{3k+2}^2$-free.
However, in view of \eqref{align11}, we get
\begin{center}
 $(n_1-1)x_1-n_3x_3=
\frac{(\rho(G)+1)\left((n_1-n_3-1)\rho(G)-n_3\right)}{(\rho(G)+n_1)(\rho(G)+n_3)}x_{u^*}
>0,
$
\end{center}
since $n_1\geq n_3+2$ and $\rho(G)>\frac23n>n_3$.
This leads to that
\begin{center}
 $\rho(G'')-\rho(G)\ge
\sum\limits_{v\in V_1\setminus (W^5\cup\{u_0\})}2x_{u_0}x_v-
   \sum\limits_{v\in V_3\setminus W^5}2x_{u_0}x_v=2x_1\big((n_1-1)x_1-n_3x_3\big)>0,
$
\end{center}
a contradiction.
Therefore, $n_1-n_3\le 1$ and $G\cong K_1+T_{n-1,3}$.
This completes the proof.
\end{proof}

\end{document}